\documentclass{article} 
\usepackage{nips14submit_e,times}

\usepackage[bookmarks=true]{hyperref}
\usepackage{graphicx,psfrag}
\usepackage[noend]{algorithmic}
\usepackage{xspace}
\usepackage{amssymb,amsmath,amscd,mathrsfs}
\usepackage{amsthm}
\usepackage{comment,color}
\usepackage{subfigure}
\usepackage{url}
\usepackage{multirow}
\usepackage{etoolbox}
\usepackage{macros}
\usepackage{xr}
\title{\LARGE \bf
Constrained Stochastic Optimal Control with a Baseline Performance Guarantee
}

\author{
Yinlam Chow\\
Institute of Computational \& Mathematical Engineering, Stanford University \\
Mohammad Ghavamzadeh\thanks{Mohammad Ghavamzadeh is at Adobe Research, on leave of absence from INRIA Lille - Team SequeL.}\\
INRIA Lille - Team SequeL \& Adobe Research}

\nipsfinalcopy 

\newcommand{\BEAS}{\begin{eqnarray*}}
\newcommand{\EEAS}{\end{eqnarray*}}
\newcommand{\BEQ}{\begin{equation}}
\newcommand{\EEQ}{\end{equation}}
\newcommand{\BIT}{\begin{itemize}}
\newcommand{\EIT}{\end{itemize}}

\newcommand{\expec}{\mathbb{E}}

\usepackage{color}
\usepackage{amsmath}
\usepackage{amssymb}
\usepackage{graphicx}
\usepackage{comment,xspace}
\usepackage{fancybox}

\newcommand{\real}{{\mathbb{R}}}
\newcommand{\reals}{\real}

\renewcommand{\natural}{{\mathbb{N}}}
\newcommand{\naturals}{\natural}

\newtheorem{proposition}[theorem]{Proposition}

\newtheorem{definition}[theorem]{Definition}

\newtheorem{assumption}[theorem]{Assumption}

\newcommand{\argmin}{\operatornamewithlimits{argmin}}

\def\A{\mathcal{A}}

\def\V{V}

\def\X{\mathcal{X}}

\begin{document}

\maketitle
\thispagestyle{empty}
\pagestyle{empty}

\begin{abstract}
In this paper, we show how a simulated Markov decision process (MDP) built by the so-called \emph{baseline} policies, can be used to compute a different policy, namely the \emph{simulated optimal} policy, for which the performance of this policy is guaranteed to be better than the baseline policy in the real environment. This technique has immense applications in fields such as news recommendation systems, health care diagnosis and digital online marketing. Our proposed algorithm iteratively solves for a ``good" policy in the simulated MDP in an offline setting. Furthermore, we provide a performance bound on sub-optimality for the control policy generated by the proposed algorithm.
\end{abstract}

\section{Introduction}\label{sec:intro}
In this paper, we show how a simulated Markov decision process (MDP) built by the so-called \emph{baseline} policies, can be used to compute a different policy, namely the \emph{simulated optimal} policy, for which the performance of this policy is guaranteed to be better than the baseline policy in the real environment. This policy evaluation technique has immense applications in various fields, where offline calculation of a new policy is computationally inexpensive, but execution of a new policy can be hazardous or costly if this new policy performs worse than the baseline policy. There are numerous applications that require such safety concerns. For example, in the personalized news article recommendation systems \cite{li2010contextual}, one requires a ``good" policy that selects articles to serve users based on contextual information in order to maximize total user clicks. In this online setup, real-time policy evaluation is dangerous in the sense that a slight fluctuation of performance may yield a huge loss in click rates. 
Similar examples can be found in patient diagnosis systems \cite{jeanpierre2004automated,goulionis2009medical} as well, where testing new clinical strategies on human trials is often too risky. Further applications on medical trials can be found in automatic Propofol administration \cite{smith1993adverse} and neuroprosthetic control \cite{mahmoudi2010extracting}. Last but not least, under the framework of digital online marketing \cite{kiang2000marketing}, one also require a ``good" ad-recommendation policy that sequentially displays attractive ads in order to increase customers' click-rates while the performance of this policy is guaranteed to be better than the company's baseline marketing strategy. 

In this paper, we make several assumptions. First, we assume the evolution of the real environment follows a Markov decision process. This assumption is standard in most literatures of sequential decision making problems \cite{bertsekas1995dynamic,puterman2009markov}. Second, we also assume that a deterministic error upper bound function between the true MDP and the simulated MDP is known. While this assumption can be justified probabilistically using the Chernoff-hoeffding inequality when the state and action spaces are finite, one may find it challenging to obtain such an error bound in advance for systems with continuous state and action spaces, especially when there is no prior knowledge in the real system model. Nevertheless, from the best of our knowledge, the method proposed in this paper is still novel and applicable in many areas. One can easily extend the analysis of our methods to include the case of probabilistic error bounds as well.

Since our problem formulation guarantees that the resultant control policy performs better than its baseline counterpart, one may connect our method to the vast literature of stochastic control problems with system uncertainties \cite{nilim2005robust,xu2010distributionally,regan2010robust}. The latter approach solves for a robust policy which guarantees worst-case performance. While our proposed method can also be posed as a stochastic control problem with system uncertainties characterized by the error upper bound, the resultant policy may be over-conservative. Furthermore, we are only interested in its performance under the real environment, rather than the worst case performance. On top of that, in the worst-case approach, various assumptions are required for the set of system uncertainties in order to construct numerically tractable algorithms. For example, see the rectangularity assumption of uncertain transition probability set in \cite{iyengar2005robust}. Therefore, our proposed algorithm in general provides a simple yet better control policy, compared to its robust counterpart. 

The rest of this paper is structured as follows. In Section \ref{sec:prelimMDP}, preliminary notations and definitions used in this paper are discussed. The constrained MDP problem formulation is given in Section \ref{sec:problem} where further analysis on the Lagrangian formulation based on state augmentations is given in Section \ref{sec:aug_MDP}. Based on the strong duality result in Section \ref{sec:strong_duality}, one can equivalently consider the dual Lagrangian formulation. On top of this, by using the feasibility condition in Section \ref{sec:suff_cond}, an iterative algorithm is provided in Section \ref{sec:algorithm}, followed by a performance analysis of sub-optimality in Section \ref{sec:perf}. Finally the results of this paper is concluded in Section \ref{sec:conclusion}.

\section{Preliminaries} 
\label{sec:prelimMDP}
We consider problems in which the agent's interaction with the environment is modeled as a MDP. A MDP is a tuple $M=(\X,\A,R,P,P_0)$, where $\X=\{1,\ldots,n\}$ and $\A=\{1,\ldots,m\}$ are the state and action spaces; $R(x,a)\in[-R_{\max},R_{\max}]$ is the bounded cost random variable whose expectation is denoted by $r(x,a)=\E\big[R(x,a)\big]$; $P(\cdot|x,a)$ is the transition probability distribution; and $P_0(\cdot)$ is the initial state distribution. For simplicity, we assume that the system has a single initial state $x^0$, i.e.,~$P_0(x)=\mathbf{1}\{x=x^0\}$. All the results of the paper can be easily extended to the case that the system has more than one initial state. 
We also need to specify the rule according to which the agent selects actions at each state. A {\em stationary policy} $\mu(\cdot|x)$ is a probability distribution over actions, conditioned on the current state. We denote by $d_\gamma^\mu(x|x^0)=(1-\gamma)\sum_{k=0}^\infty\gamma^k \mathbb P(x_k=x|x_0=x^0;\mu)$ and $\pi_\gamma^\mu(x,a|x^0)=d_\gamma^\mu(x|x^0)\mu(a|x)$ the $\gamma$-discounted visiting distribution of state $x$ and state-action pair $(x,a)$ under policy $\mu$, respectively. 

However, even if the state/action spaces and reward functions are given, in many cases the real world model $M$ is unknown, in particular, the underlying transition probability $P$ cannot be easily obtained by the user. Rather, a simulated transition probability $\widehat{P}$ is often obtained by Monte Carlo sampling techniques with a baseline policy distribution $\mu_B$. In many engineering applications, the following simulated MDP model $\widehat{M}=(\X,\A,R,\widehat{P},P_0)$ can be built based on previous histories of data and Monte Carlo sampling techniques. There are many approaches to search for a ``good" control policy using the information of MDP model $\widehat{M}$. As the most direct approach, one can search for an optimal policy over $\widehat{M}$ with no interactions with the real world using dynamic programming or approximate dynamic programming methods. Unfortunately, while it could be computationally expensive to find an optimal policy from $\widehat{M}$ when the state/action spaces are huge, this optimal policy in general do not have a performance guarantee to the original model $M$. 

In order to characterize the error between the true and simulated models, define the ``mis-measure" of transition probability at $(x,a)\in\mathcal X\times \mathcal A$ as the 1-norm of the difference between $P$ and $\widehat{P}$, i.e.,
\[
\|P(\cdot|x,a)-\widehat{P}(\cdot|x,a)\|_1=\sum_{y\in\mathcal X}\left|P(y|x,a)-\widehat{P}(y|x,a)\right|.
\]
Suppose an upper bound $e(x,a)$ of the mis-measure is given, i.e.,
\[
e(x,a)\geq \sum_{y\in\mathcal X}\left|P(y|x,a)-\widehat{P}(y|x,a)\right|.
\]
Using the notion of transition probability ``mis-measure" , we will later illustrate a policy search method in this paper that finds an optimal control policy in the simulated model with guaranteed performance improvements over the baseline policy $\mu_B$. 

Before getting into the main results, we first introduce several policy spaces.
Define the space $H_t$ of admissible histories up to time $t$ by $H_t = H_{t-1} \times \mathcal X\times \mathcal A$, for $t\geq 1$, and $H_0=\mathcal X$. A generic element $h_t\in H_t$ is of the form $h_t = (x_0, a_0, \ldots , x_{t-1}, a_{t-1}, x_t)$. Let $\Pi_t$ be the set of all stochastic history dependent policies with the property that at each time $t$ the control is a function of $h_t$. In other words, $\Pi_{H,t} := \Bigl \{ \{\mu_0: H_0 \rightarrow \mathbb P(\mathcal A),\, \mu_1: H_1 \rightarrow \mathbb P(\mathcal A), \ldots,\mu_{t-1}: H_{t-1} \rightarrow \mathbb P(\mathcal A)\} | \mu_k(h_k) \in \mathbb P(\mathcal A) \text{ for all } h_k\in H_k, \, t-1\geq k\geq 0 \Bigr\}$. While history dependent policy space is the most general space for control policies that satisfy the causality assumption, it is often times computationally intractable to search for an optimal policy over this space. On the other hand, we define the set of Markov policies as a sequence of state-action mappings as follows: $\Pi_{M,t} := \Bigl \{ \{\mu_0: \mathcal X \rightarrow \mathbb P(\mathcal A),\, \mu_1: \mathcal X \rightarrow \mathbb P(\mathcal A), \ldots,\mu_{t-1}: \mathcal X \rightarrow \mathbb P(\mathcal A)\} | \mu_k(x_k) \in \mathbb P(\mathcal A) \text{ for all } x_k\in \mathcal X, \, t-1\geq k\geq 0 \Bigr\}$. Furthermore, when the state-action mapping is stationary across time, we define the space of stationary Markov policies as follows $\Pi_{S,t} := \Bigl \{ \{\mu: \mathcal X \rightarrow \mathbb P(\mathcal A),\, \mu: \mathcal X \rightarrow \mathbb P(\mathcal A), \ldots,\mu: \mathcal X \rightarrow \mathbb P(\mathcal A)\} | \mu(x_k) \in \mathbb P(\mathcal A) \text{ for all } x_k\in \mathcal X, \, t-1\geq k\geq 0 \Bigr\}$. We also let $\Pi_{H}=\lim_{t\rightarrow\infty}\Pi_{H,t}$, $\Pi_{M}=\lim_{t\rightarrow\infty}\Pi_{M,t}$ and $\Pi_{S}=\lim_{t\rightarrow\infty}\Pi_{S,t}$ be the set of history dependent policies, Markov policies and stationary Markov policies respectively. We will later show that the proposed method performs policy search over the space of stationary Markov policies without loss of optimality.

\section{Problem Formulation}\label{sec:problem}
We propose the following stochastic optimal control problem whose optimal policy $\mathcal U^\ast=\{\mu_{0}^\ast,\mu_{1}^\ast,\ldots\}$, 1) maximizes the expected return in model based MDP $\widehat{M}$ and 2) guarantees a better performance then $\mathcal U^B=\{\mu_{0}^B,\mu_{1}^B,\ldots\}$ in the real world MDP $M$,
\begin{alignat*}{2}
\max_{\mathcal U\in\Pi_H}& &\quad &\mathbb E_{\widehat{P}}\left(\sum_{k=0}^\infty\gamma^k r(x_k,a_k)\mid a_k\sim\mu_k,P_0\right)\\  
\text{subject to} & & \quad & \mathbb E_{P}\left(\sum_{k=0}^\infty\gamma^k r(x_k,a_k)\mid a_k\sim\mu_k,P_0\right)\geq \mathbb E_P\left(\sum_{k=0}^\infty\gamma^k r(x_k,a_k)\mid a_k\sim\mu_{k}^B,P_0\right).
\end{alignat*}
Now define the Lagrangian function as
\[
\begin{split}
L(\mathcal U,\lambda)\stackrel{\triangle}{=}&\mathbb E_{\widehat{P}}\left(\sum_{k=0}^\infty\gamma^k r(x_k,a_k)\mid a_k\sim\mu_k,P_0\right) \\
&+\lambda\left(\mathbb E_P\left(\sum_{k=0}^\infty\gamma^k r(x_k,a_k)\mid a_k\sim\mu_k,P_0\right)-  \mathbb E_P\left(\sum_{k=0}^\infty\gamma^k r(x_k,a_k)\mid a_k\sim\mu_{k}^B,P_0\right)\right).
\end{split}
\]
Since the baseline policy $\mathcal U^B$ is given, we define 
\[
M_B=\mathbb E_P\left(\sum_{k=0}^\infty\gamma^k r(x_k,a_k)\mid a_k\sim\mu_{k}^B,P_0\right)
\] 
as the safety threshold for the above problem and assume this quantity is known to the user. To solve the constrained risk sensitive optimization problem, we employ the Lagrangian formulation to convert the above problem into the following unconstrained problem:  
\begin{equation}
\label{eq:unconstrained-discounted-risk-measure}
\max_{\mathcal U\in\Pi_H}\min_{\lambda\geq 0} L(\mathcal U,\lambda).
\end{equation}
\section{The Augmented MDP}\label{sec:aug_MDP}
Since the Lagrangian function involves multiple probability distributions, the above problem cannot be easily solved using existing stochastic control methods from the MDP literature. Therefore, one constructs the following augmented MDP in order to rewrite $L(\mu,\lambda)$ under one MDP. Consider an augmented MDP $M^{\text{aug}}_{\lambda}=(\bar{\mathcal{X}},\bar{\mathcal{A}},  P^{\text{aug}},r_{\lambda},\gamma,P^{\text{aug}}_0)$ such that $\bar{\mathcal{X}}=\mathcal X\times I$, where $I=\{0,1\}$, $\bar{\mathcal{A}}=\mathcal A$, $\gamma$ is the discounted factor given above and,
\[
\begin{split}
r_{\lambda}(x,i,a)&=\left\{
\begin{array}{cc}
2 r(x,a)\mathbf{1}\{x=x_T\}&\text{if $i=0$}\\
2 \lambda r(x,a)\mathbf{1}\{x=x_T\}&\text{if $i=1$}\\
\end{array}\right.,\,\,\forall (x,i)\in\bar{\mathcal{X}},\,\, a\in\bar{\mathcal{A}},\\
P^{\text{aug}}_0(x,i)&=\left\{\begin{array}{cc}
\frac{1}{2} P_0(x)&\text{if $i=0$}\\
\frac{1}{2} P_0(x)&\text{if $i=1$}\\
\end{array}\right.,\,\,\forall (x,i)\in\bar{\mathcal{X}}.
\end{split}
\]
We also define the uncertain transition probability as follows:
\[
P^{\text{aug}}(x^\prime,i^\prime|x,i,a)=\left\{
\begin{array}{cc}
\widehat{P}(x^\prime|x,a)\mathbf 1\{i^\prime=0\}&\text{if $i=0$}\\
P(x^\prime|x,a)\mathbf 1\{i^\prime=1\}&\text{if $i=1$}\\
\end{array}\right.,\,\,\forall (x^\prime,i^\prime),(x,i)\in\bar{\mathcal{X}}.
\]
For any $(x,i)\in\bar{\mathcal X}$, define
the robust $\lambda-$parametrized optimal Bellman's operator $T_{\lambda}:\reals^{|\bar{\mathcal X}|}\rightarrow\reals^{|\bar{\mathcal X}|}$ as follows:
                     \[
                      T_{\lambda}[V](x,i)=\max_{a\in\bar{\mathcal A}}\left\{r_{\lambda}(x,i,a)+\gamma\sum_{(x^\prime,i^\prime)\in\bar{\mathcal X}}P^{\text{aug}}(x^\prime,i^\prime|x,i,a)V(x^\prime,i^\prime)\right\}.
                      \]
In the following analysis, we will assume $\lambda\geq 0$ is bounded\footnote{The boundedness assumption is justified in the solution algorithm section.}. 
From the literature in dynamic programming, one can easily see that the above Bellman operator satisfies the following properties, and the fixed point solution converges to the optimal solution of the Lagrangian function. For the proofs of these properties, please see Chapter 1 of \cite{bertsekas1995dynamic} for more details.

\begin{proposition}
The Bellman operator $T_{\lambda}[\V]$ has the following properties:
\begin{itemize}
\item (Monotonicity) If $\V_1(x,i)\geq \V_2(x,i)$, for any $x\in\X$, $i\in I$ then $T_{\lambda}[\V_1](x,i)\geq T_{\lambda}[\V_2](x,i)$.
\item (Constant shift) For $K\in\reals$, $T_{\lambda}[\V+K](x,i)=T_{\lambda}[\V](x,i)+\gamma K$.
\item (Contraction) 
For any $V_1,V_2:\mathcal X\times I\rightarrow\reals$, $\|T_{\lambda}[\V_1]-T_{\lambda}[\V_2]\|_{\infty}\leq \gamma \|\V_1-\V_2\|_{\infty}$, 
where $\|f\|_{\infty}=\max_{x\in\X,i\in I} |f(x,i)|$.
\end{itemize}
\end{proposition}
\begin{theorem}\label{lem:tech_2}
For any $\lambda\geq 0$, there exists a unique solution to the fixed point equation: $T_\lambda[V](x,i)=V(x,i)$, $\forall (x,i)\in\bar{\mathcal X}$. Let $V_\lambda\in\reals^{|\bar{\mathcal X}|}$ be such unique fixed point solution. For any initial value function $V_{\lambda,0}(x,i)$, $V_\lambda(x,i)=\lim_{N\rightarrow\infty} T^N_\lambda[V][V_{\lambda,0}](x,i)$, $\forall (x,i)\in\bar{\mathcal X}$. Furthermore, 
\begin{equation}\label{eq:opt_cond}
\sum_{x,i}P^{\text{aug}}_0(x,i)V_\lambda(x,i)= \max_{\mathcal U\in \Pi_H}L(\mathcal U,\lambda)+\lambda M_B.
\end{equation}
\end{theorem}
By the above theory of Bellman's recursion, one can show that the space of augmented state feedback Markov policies is dominating. 
\begin{theorem}[Dominating Policies]
Suppose $V_\lambda$ is the unique solution to $T_\lambda[V](x,i)=V(x,i)$, $\forall (x,i)\in\bar{\mathcal X}$. Then, suppose the policy $\mu$ is found by
\[
\mu(x,i)\in\arg\max_{a\in\bar{\mathcal A}}\left\{r_{\lambda}(x,i,a)+\gamma\sum_{(x^\prime, i^\prime)\in\bar{\mathcal X}}P^{\text{aug}}(x^\prime,i^\prime|x,i,a)V(x^\prime,i^\prime)\right\}.
\]
Then for $\mathcal U^\ast=\{\mu,\mu,\ldots\}$ being a sequence of Markov stationary optimal policies, it follows that $\mathcal U^\ast\in\arg\max_{\mathcal U\in\Pi_H}L(\mathcal U,\lambda)+\lambda M_B$. In other words, the class of $\lambda-$parametrized Markov stationary deterministic policies is dominating.
\end{theorem}
\begin{proof}
Suppose $T^{\mu}_{\lambda}[V_\lambda](x,i)=T_{\lambda}[V_\lambda](x,i)$ for any $(x,i)\in\bar{\mathcal X}$, where the Bellman operator $T^{\mu}_{\lambda}[V]$ is defined as follows:
\[
 T^{\mu}_{\lambda}[V](x,i)=r_{\lambda}(x,i,\mu(x,i))+\gamma\sum_{(x^\prime,i^\prime)\in\bar{\mathcal X}}P^{\text{aug}}(x^\prime,i^\prime|x,i,\mu(x,i))V(x^\prime,i^\prime).
\]
Then, from the above equality and the fixed point equation, $T^{\mu}_{\lambda}[V_\lambda](x,i)=T_{\lambda}[V_\lambda](x,i)=V_\lambda(x,i)$. By the unique-ness of the fixed point equation: $T^{\mu}_{\lambda}[V](x,i)=V(x,i)$, one also obtains $V_\lambda(x,i)=V^\mu_{\lambda}(x,i)$, where $V^\mu_{\lambda}(x,i)=\mathbb E_{P^{\text{aug}}}\left(\sum_{k=0}^\infty\gamma^k r_{\lambda}(x_k,i_k,a_k)\mid a_k\sim\mu,x_0=x,i_0=i\right)$. By summing over $(x,i)\in\bar{\mathcal X}$, weighted by the initial distribution $\{P^{\text{aug}}_0(x,i)\}_{(x,i)\in\bar{\mathcal X}}$, one obtains 
\[
\begin{split}
\max_{\mathcal U\in\Pi_H}L(\mathcal U,\lambda)+\lambda M_B=&\mathbb E_{P^{\text{aug}}}\left(\sum_{k=0}^\infty\gamma^k r_{\lambda}(x_k,i_k,a_k)\mid a_k\sim\mu,P^{\text{aug}}_0\right)\\
=&\sum_{x,i}P^{\text{aug}}_0(x,i)V^\mu_\lambda(x,i)=\sum_{x,i}P^{\text{aug}}_0(x,i)V_\lambda(x,i),
\end{split}
\]
which further implies $\mu^\ast$ is a sequence of Markov stationary optimal policies (and $\mu$ is optimal).
\end{proof}

Since the class of stationary Markov policies is dominating, without loss of generality, one can write
\[
\max_{\mathcal U\in\Pi_H}\min_{\lambda\geq 0} L(\mathcal U,\lambda)=\max_{\mu\in\Pi_S}\min_{\lambda\geq 0} L(\mu,\lambda),
\]
where 
\[
\begin{split}
L(\mu,\lambda)=&\mathbb E_{\widehat{P}}\left(\sum_{k=0}^\infty\gamma^k r(x_k,a_k)\mid a_k\sim\mu,P_0\right) +\lambda\mathbb E_{P}\left(\sum_{k=0}^\infty\gamma^k r(x_k,a_k)\mid a_k\sim\mu,P_0\right)-\lambda M_B.
\end{split}
\]
and $\mu$ belongs to the set of stationary Markov policies of the augmented MDP $M^{\text{aug}}_{\lambda}$.

\section{Strong Duality}\label{sec:strong_duality}
In previous sections, one transforms the constrained optimization problem into its primal Lagrangian formulation. However in often times, the dual Lagrangian formulation often yields a computationally traceable and efficient algorithm. In this section, we want to show that strong duality exists, thus one can equivalently switch from the primal formulation to its dual counterpart.
Consider the primal optimization problem: 
\[
\mathcal P(P^{\text{aug}}_0,M_B)=\max_{\mu}\min_{\lambda\geq 0}\,\mathbb E_{  P^{\text{aug}}}\left(\sum_{k=0}^\infty\gamma^k r_{\lambda}(x_k,i_k,a_k)\mid a_k\sim\mu, P^{\text{aug}}_0\right)-\lambda M_B= \max_{\mu}\min_{\lambda\geq 0} L(\mu,\lambda),
\]
and the dual problem:
\[
\mathcal D(P^{\text{aug}}_0,M_B)=\min_{\lambda\geq 0}\max_{\mu}\,\mathbb E_{  P^{\text{aug}}}\left(\sum_{k=0}^\infty\gamma^k r_{\lambda}(x_k,i_k,a_k)\mid a_k\sim\mu, P^{\text{aug}}_0\right)-\lambda M_B=\min_{\lambda\geq 0}\max_{\mu} L(\mu,\lambda).
\]
Now consider the risk neutral optimization problem. Recall from that previous section that
\[
\max_{\mu}\mathbb E_{  P^{\text{aug}}}\left(\sum_{k=0}^\infty\gamma^k r_{\lambda}(x_k,i_k,a_k)\mid a_k\sim\mu, P^{\text{aug}}_0\right)=\sum_{(x,i)\in\bar{ \mathcal X}}P^{\text{aug}}_0(x,i)V_\lambda(x,i),
\]
where $V_\lambda(x,i)$ is the unique solution of the following Bellman equation $T_\lambda[V](x,i)=V(x,i)$.
Furthermore, by the linear programming formulation of Bellman equation, one obtains 
\begin{alignat}{2}
\text{RLP}(\lambda)=&\min_{V}&\quad & \sum_{(x,i)\in\bar{ \mathcal X}}P^{\text{aug}}_0(x,i)V(x,i)\label{problem_2}\\
&\text{s.t.} &\quad &r_{\lambda}(x,i,a)+\gamma\sum_{(x^\prime,i^\prime)\in\bar{\mathcal X}} P^{\text{aug}}(x',i'|x,i,a)V(x^\prime,i^\prime)\leq V(x,i),\,\,\forall (x,i)\in\bar{\mathcal X}, \,\,a\in \bar{\mathcal A}\nonumber\\
=&\max_{\mu}&\quad &\,\mathbb E_{P^{\text{aug}}}\left(\sum_{k=0}^\infty\gamma^k r_{\lambda}(x_k,i_k,a_k)\mid a_k\sim\mu, P^{\text{aug}}_0\right).\label{eq:1}
\end{alignat}

Consider the Lagrangian of the above problem with Lagrangian variable $\{\beta(x,i,a)\}_{(x,i)\in\bar{\mathcal X},a\in\bar{\mathcal A}}\in\reals^{|\bar{\mathcal X}|\times |\bar{\mathcal A}|}_{\geq 0}$:
\[
\begin{split}
\mathcal L_\lambda(V,\beta)=\sum_{(x,i)\in\bar{ \mathcal X}}P^{\text{aug}}_0(x,i)&V(x,i)+\sum_{(x,i)\in\bar{ \mathcal X},a\in\bar{\mathcal A}}\beta(x,i,a)\cdot\\
&\left\{r_{\lambda}(x,i,a)+\gamma\sum_{x'\in\mathcal X}P^{\text{aug}}(x',i'|x,i,a)V(x',s',i')- V(x,i)\right\},
\end{split}
\]
where
\[
\min_{V}\max_{\beta\in\reals^{|\bar{\mathcal X}|\times |\bar{\mathcal A}|}_{\geq 0}}\mathcal L_\lambda(V,\beta)=\max_{\mu}\mathbb E_{  P^{\text{aug}}}\left(\sum_{k=0}^\infty\gamma^k r_{\lambda}(x_k,i_k,a_k)\mid a_k\sim\mu, P^{\text{aug}}_0\right)=\text{RLP}(\lambda)
\]
follows from the primal argument of Lagrangian duality theory.
Note that $\mathcal L_\lambda(V,\beta)$ is a concave function in $\beta$ and convex function in $V$. Assume Slater's condition holds, this implies the following result:
\[
\min_{V}\max_{\beta\in\reals^{|\bar{\mathcal X}|\times |\bar{\mathcal A}|}_{\geq 0}}\mathcal L_\lambda(V,\beta)=\max_{\beta\in\reals^{|\bar{\mathcal X}|\times |\bar{\mathcal A}|}_{\geq 0}}\min_{V}\mathcal L_\lambda(V,\beta).
\]
Before proving the main result, we first state the Aubin's Minmax theorem.
\begin{lemma}[Minmax Theorem \cite{aubin1972pareto}]
Let $G_1$ and $G_2$ be convex subsets of linear topological spaces, and let $G_1$ be compact. Consider a function $\psi : G_1 \times G_2 \rightarrow \reals$ such that
\begin{itemize}
 \item for each $g_2 \in G_2$, $g_1 \rightarrow \psi(g_1, g_2)$ is convex and lower semi-continuous, and
\item for each $g_1 \in G_1$, $g_2 \rightarrow \psi(g_1, g_2)$ is concave.
\end{itemize}
Then there exists some $g_1^*\in G_1$ such that
\[
\inf_{G_1}\sup_{G_2}\psi(g_1,g_2)=\sup_{G_2}\psi(g^*_1,g_2)=\sup_{G_2}\inf_{G_1}\psi(g_1,g_2).
\]
\end{lemma}

Then, the next theorem shows that strong duality holds.
\begin{theorem}
For any given initial distribution $P^{\text{aug}}_0$ and safety threshold $M_B$, the following statement holds:
\[
\mathcal D(P^{\text{aug}}_0,M_B)=\mathcal P(P^{\text{aug}}_0,M_B).
\]
\end{theorem}
\begin{proof}
First consider the following expression: 
\[
\begin{split}
&\min_{\lambda\geq 0}\max_{\mu}\mathbb E_{  P^{\text{aug}}}\left(\sum_{k=0}^\infty\gamma^k r_{\lambda}(x_k,i_k,a_k)\mid a_k\sim\mu, P^{\text{aug}}_0\right)\\
=&\min_{V,\lambda\geq 0}\max_{\beta\in\reals^{|\bar{\mathcal X}|\times |\bar{\mathcal A}|}_{\geq 0}}\mathcal L_\lambda(V,\beta)=\min_{\lambda\geq 0}\max_{\beta\in\reals^{|\bar{\mathcal X}|\times |\bar{\mathcal A}|}_{\geq 0}}\min_{V}\mathcal L_\lambda(V,\beta).
\end{split}
\]
The first equality follows from the previous primal arguments of duality theory and the second equality follows from strong duality. Also, consider the following function
\[
\Psi(\lambda,\beta):=\min_{V}\mathcal L_\lambda(V,\beta)= \sum_{(x,i)\in\bar{ \mathcal X},a\in\bar{\mathcal A}}\beta(x,i,a)r_{\lambda}(x,i,a)+\Phi(\beta),
\]
where
\[
\begin{split}
\Phi(\beta)=\min_V&\sum_{x\in\mathcal X}P^{\text{aug}}_0(x,0)V(x,0)+\sum_{x\in \mathcal X,a\in\bar{\mathcal A}}\beta(x,0,a)\left\{ \gamma\sum_{x'\in\mathcal X}\widehat{P}(x'|x,a)V(x',0)- V(x,0)\right\}\\
&+\sum_{x\in\mathcal X}P^{\text{aug}}_0(x,1)V(x,1)+\sum_{x\in \mathcal X,a\in\bar{\mathcal A}}\beta(x,1,a)\left\{ \gamma\sum_{x'\in\mathcal X}P(x'|x,a)V(x',1)- V(x,1)\right\}.
\end{split}
\]
This implies that whenever $\Phi(\beta)$ is bounded, the following condition holds:
\begin{alignat}{1}
P^{\text{aug}}_0(x,1)=&\sum_{a\in\bar{\mathcal A}}\beta(x,1,a)\sum_{x'\in\mathcal X}\delta\{x=x^\prime\}- \gamma P(x'|x,a),\,\,\forall x\in\mathcal X, \label{eq:i_1}\\
P^{\text{aug}}_0(x,0)=&\sum_{a\in\bar{\mathcal A}}\beta(x,0,a)\sum_{x'\in\mathcal X}\delta\{x=x^\prime\}- \gamma \widehat{P}(x'|x,a),\,\,\forall x\in\mathcal X.\label{eq:i_0}
\end{alignat}
Then, one concludes that
\[
\Phi(\beta)=\left\{\begin{array}{cc}
0 & \text{if $P^{\text{aug}}_0(x,i)=\sum_{a\in\bar{\mathcal A}}\beta(x,i,a)\sum_{x',i'}\delta\{x=x^\prime,i=i^\prime\}- \gamma P^{\text{aug}}(x',i'|x,i,a),\,\,\forall x,i$}\\
-\infty & \text{otherwise}
\end{array}\right.
\]
and
\[
\Psi(\lambda,\beta)=\left\{\begin{array}{cc}
 \sum_{x,i,a}\beta(x,i,a)r_{\lambda}(x,i,a)/(1-\gamma) & \text{if $\beta\in\mathcal B$}\\
 -\infty &\text{otherwise} 
 \end{array}\right.
\]
where
\[
\mathcal B=\left\{\beta\in\reals^{|\bar{\mathcal X}|\times |\bar{\mathcal A}|}_{\geq 0}:(1-\gamma)P^{\text{aug}}_0(x,i)=\sum_{a\in\bar{\mathcal A}}\beta(x,i,a)\sum_{x',i'}\delta\{x=x^\prime,i=i^\prime\}- \gamma P^{\text{aug}}(x',i'|x,i,a),\,\,\forall x,i\right\}.
\]
Note that for any $\beta\in\mathcal B$, $\sum_{x,i,a}\beta(x,i,a)=1$.
Since $\Psi(\lambda,\beta)-\lambda M_B$ is a linear function in $\lambda$, it is also a convex and lower semi-continuous function in $\lambda$. Furthermore $\{\lambda\geq 0\}$ is a convex compact set of $\lambda$. On the other hand, $\Psi(\lambda,\beta)$ is a concave function in $\beta$, due to the facts that 1)  $\sum_{(x,i)\in\bar{ \mathcal X},a\in\bar{\mathcal A}}\beta(x,i,a)r_{\lambda}(x,i,a)/(1-\gamma)$ is a linear function of $\beta$ and 2) $\mathcal B$ is a convex set. Therefore, by Aubin's Minimax Theorem, one concludes that
\[
\begin{split}
\min_{\lambda\geq 0}\max_{\beta\in\mathcal B}\min_{V}\mathcal L_\lambda(V,\beta)-\lambda M_B=&\min_{\lambda\geq 0}\max_{\beta\in\mathcal B}\Psi(\lambda,\beta)-\lambda M_B\\
=&\max_{\beta\in\mathcal B}\min_{\lambda\geq 0}\Psi(\lambda,\beta)-\lambda M_B=\max_{\beta\in\mathcal B}\min_{\lambda\geq 0}\min_{V}\mathcal L_\lambda(V,\beta)-\lambda M_B.
\end{split}
\]

The final step is to prove 
\[
\max_{\beta\in\mathcal B}\min_{\lambda\geq 0}\sum_{(x,i)\in\bar{ \mathcal X},a\in\bar{\mathcal A}}\beta(x,i,a)\frac{r_{\lambda}(x,i,a)}{1-\gamma}-\lambda M_B=\max_{\mu}\min_{\lambda\geq 0}\,\mathbb E_{  P^{\text{aug}}}\left(\sum_{k=0}^\infty\gamma^k r_{\lambda}(x_k,i_k,a_k)\mid a_k\sim\mu, P^{\text{aug}}_0\right)-\lambda M_B.
\]
The proof of this statement follows directly from Theorem 3.1 in \cite{altman1999constrained}. 
\end{proof}
We have just shown that strong duality holds in this optimization problem. Thus, in the following sections, we will mainly focus on deriving algorithms to solve the dual problem described in $\mathcal D(P^{\text{aug}}_0,M_B)$.

\section{A Condition for Guaranteeing Feasibility}\label{sec:suff_cond}
Recall the risk neutral maximization problem is equivalent to its primal Lagrangian formulation 
\[
\max_{\mu} \min_{\lambda\geq 0} L(\mu,\lambda).
\]
However when the MDP $M$ is unknown, one cannot directly solve this problem. Therefore, we want to derive a sufficient condition to guarantee feasibility based on solving a similar problem using MDP $\widehat{M}$. Note that the risk neutral maximization problem is feasible if $\max_{\mu} \min_{\lambda\geq 0} L(\mu,\lambda)$ is lower bounded. Therefore if there exists a function $\widehat{L}(\mu,\lambda)$ such that $\widehat{L}(\mu,\lambda)\leq  L(\mu,\lambda)$, then 
\[
\min_{ \lambda\geq 0} L(\mu,\lambda):=L(\mu,\lambda^*)\geq  \widehat{L}(\mu,\lambda^*)\geq \min_{\lambda\geq 0} \widehat{L}(\mu,\lambda).
\]
By taking maximization over $\mu$ on both sides, one further obtains
\[
\max_{\mu}\min_{ \lambda\geq 0} L(\mu,\lambda)\geq \max_{\mu}\min_{\lambda\geq 0} \widehat{L}(\mu,\lambda).
\]
Therefore, if $\max_{\mu}\min_{\lambda\geq 0}\widehat{L}(\mu,\lambda)$ is lower bounded, one concludes that the original problem is always feasible. Now, consider the following construction of $\widehat{L}(\mu,\lambda)$.

\begin{lemma}
Define
\[
\widehat{L}(\mu,\lambda):=\mathbb E_{{\widehat{P}}^{\text{aug}}}\left(\sum_{k=0}^\infty\gamma^k \widehat{r}_{\lambda}(x_k,i_k,a_k)\mid a_k\sim\mu, P^{\text{aug}}_0\right)-\lambda M_B ,
\]
where
\[
\widehat{r}_{\lambda}(x,i,a)=\left\{
\begin{array}{cc}
2 r(x,a)&\text{if $i=0$}\\
2 \lambda r(x,a)-\frac{  2\gamma\lambda R_{\max}}{1-\gamma}e(x,a)&\text{if $i=1$}\\
\end{array}\right.,\,\,\forall (x,i)\in\bar{\mathcal{X}},\,\, a\in\bar{\mathcal{A}}.
\]
Notice that $\widehat{P}^{\text{aug}}$ is the transition probability in the augmented state and action spaces such that its definition is similar to the definition of ${P}^{\text{aug}}$, except by replacing $P$ by $\widehat{P}$.

Then for any stationary Markovian policy $\mu$ and Lagrangian multiplier $\lambda\geq 0$, the following inequality holds: 
\[
\begin{split}
 &L(\mu,\lambda)-\widehat{L}(\mu,\lambda)\\
 =&\mathbb E_{  P^{\text{aug}}}\left(\sum_{k=0}^\infty\gamma^k r_{\lambda}(x_k,i_k,a_k)\mid a_k\sim\mu, P^{\text{aug}}_0\right)-\mathbb E_{{\widehat{P}}^{\text{aug}}}\left(\sum_{k=0}^\infty\gamma^k  \widehat{r}_{\lambda}(x_k,i_k,a_k)\mid a_k\sim\mu, P^{\text{aug}}_0\right)\geq  0.
 \end{split}
\]
\end{lemma}
\begin{proof}
For any transition probability $P^{\text{aug}}$ and any Markov policy $\mu$ of the augmented MDP, define 
\[
d^\mu_{ P^{\text{aug}},\gamma}(x,i)=\sum_{k=0}^\infty \gamma^k(1-\gamma)\sum_{a\in\mathcal A}\mathbb P_{   P^{\text{aug}}}(x_k=x,i_k=i|x_0=x^0,i_0=i^0,\mu).
\]
where the occupation measure satisfies the following expression:
\begin{equation}\label{eq:feas_1}
\begin{split}
&\gamma\sum_{(x^\prime,i^\prime)\in \bar{\mathcal{X}}}d^\mu_{ P^{\text{aug}},\gamma}(x',i')\sum_{a^\prime\in\bar{\mathcal A}}
P^{\text{aug}}(x,i|x^\prime,i^\prime,a^\prime)\mu(a^\prime|x^\prime,i^\prime)\\
=&(1-\gamma)\sum_{(x^\prime,i^\prime)\in \bar{\mathcal{X}}}\sum_{a^\prime\in \bar{\mathcal A}}\sum_{k=0}^\infty \gamma^{k+1}\mathbb P_{  P^{\text{aug}}}(x_k=x',i_k=i'|x_0=x^0,i_0=i^0,\mu)  P^{\text{aug}}(x,i|x^\prime,i^\prime,a^\prime)\mu(a^\prime|x^\prime,i^\prime)\\
=&(1-\gamma)\sum_{k=0}^\infty \gamma^{k+1}\mathbb P_{  P^{\text{aug}}}(x_{k+1}=x,i_{k+1}=i|x_0=x^0,i_0=i^0,\mu)\\
=&(1-\gamma)\sum_{k=0}^\infty \gamma^k\mathbb P_{  P^{\text{aug}}}(x_{k}=x,i_k=i|x_0=x^0,i_0=i^0,\mu)-(1-\gamma)\mathbf 1\{x_0=x,i_0=i\}\\
=&d^\mu_{ P^{\text{aug}},\gamma}(x,i)-(1-\gamma)\mathbf 1\{x_0=x,i_0=i\}.
\end{split}
\end{equation}
Based on the definition of occupation measures, one can easily write
\[
\mathbb E_{ P^{\text{aug}}}\left(\sum_{k=0}^\infty\gamma^k r_{\lambda}(x_k,i_k,a_k)\mid a_k\sim\mu,x_0=x^0,i_0=i^0\right)=\sum_{(x',i')\in\bar{\mathcal X}}d^\mu_{ P^{\text{aug}},\gamma}(x',i')\sum_{a'\in\bar{\mathcal A}}\mu(a'|x',i')r_\lambda(x',i',a').
\]
Define the transition probability matrix $  P^{\text{aug}}_{\mu}$ at the $\{(x,i),(x^\prime,i^\prime)\}$ as 
\[
  P^{\text{aug}}_{\mu}((x,i),(x^\prime,i^\prime)):=   P^{\text{aug}}(x^\prime,i^\prime|x,i,\mu(x,i)),
\]
and notice that
\[
(I-\gamma  P^{\text{aug}}_{\mu})\sum_{k=0}^\infty\left(\gamma  P^{\text{aug}}_{\mu}\right)^k=\sum_{k=0}^\infty\left(\gamma  P^{\text{aug}}_{\mu}\right)^k-\left(\gamma  P^{\text{aug}}_{\mu}\right)^{k+1}=I-\lim_{k\rightarrow\infty}\left(\gamma  P^{\text{aug}}_{\mu}\right)^{k}=I.
\]
Therefore, the above equation implies $(I-\gamma  P^{\text{aug}}_{\mu})^{-1}=\sum_{t=0}^\infty\left(\gamma  P^{\text{aug}}_{\mu}\right)^k<\infty$. The series is summable because $\sigma(\gamma  P^{\text{aug}}_{\mu})<1$. Furthermore, one obtains
\[
\left(I-\gamma  P^{\text{aug}}_{\mu}\right)^{-1}((x,i),(x^\prime,i^\prime))=\sum_{k=0}^\infty \gamma^k\mathbb P(x_{k}=x',i_{k}=i'|x_0=x,i_0=i,\mu).
\]
Define $d^\mu_{ P^{\text{aug}},\gamma}$ as the vector $\{d^\mu_{ P^{\text{aug}},\gamma}(x,i)\}_{(x,i)\in\bar{\mathcal X}}$, expression \eqref{eq:feas_1} can be rewritten as follows:
\[
\left(I-\gamma  P^{\text{aug}}_{\mu}\right)^\top d^\mu_{ P^{\text{aug}},\gamma}=\{\mathbf 1\{x_0=x,i_0=i\}\}_{(x,i)\in\bar{\mathcal X}}.
\] 
On the other hand, by defining ${\widehat{P}}^{\text{aug}}_{\mu}((x,i),(x^\prime,i^\prime)):= {\widehat{P}}^{\text{aug}}(x^\prime,i^\prime|x,i,\mu(x,i))$, we can also write
\[
\left(I-\gamma{\widehat{P}}^{\text{aug}}_{\mu}\right)^\top d_{{\widehat{P}}^{\text{aug}},\mu}=\{\mathbf 1\{x_0=x,i_0=i\}\}_{(x,i)\in\bar{\mathcal X}}.
\]
Combining the above expressions implies
\[
d_{P^{\text{aug}},\mu}-d_{{\widehat{P}}^{\text{aug}},\mu}-\gamma \left((P^{\text{aug}}_\mu)^\top d_{P^{\text{aug}},\mu}-({\widehat{P}}^{\text{aug}}_\mu)^\top d_{{\widehat{P}}^{\text{aug}},\mu}\right)=0,
\]
which further implies
\[
\begin{split}
& \left(I-\gamma  P^{\text{aug}}_{\mu}\right)^\top\left(d_{P^{\text{aug}},\mu}-d_{{\widehat{P}}^{\text{aug}},\mu}\right)=\gamma \left((P^{\text{aug}}_\mu)^\top -({\widehat{P}}^{\text{aug}}_\mu)^\top \right)d_{{\widehat{P}}^{\text{aug}},\mu}\\
\implies & \left(d_{P^{\text{aug}},\mu}-d_{{\widehat{P}}^{\text{aug}},\mu}\right)=\left(I-\gamma  P^{\text{aug}}_{\mu}\right)^{-\top}\gamma \left((P^{\text{aug}}_\mu)^\top -({\widehat{P}}^{\text{aug}}_\mu)^\top \right)d_{{\widehat{P}}^{\text{aug}},\mu}.
\end{split}
\]

For any policy $\mu$, one notices that
\[
\begin{split}
&(1-\gamma)\left(\mathbb E_{  P^{\text{aug}}}\left(\sum_{k=0}^\infty\gamma^k r_{\lambda}(x_k,i_k,a_k)\mid a_k\sim\mu, P^{\text{aug}}_0\right)-\mathbb E_{{\widehat{P}}^{\text{aug}}}\left(\sum_{k=0}^\infty\gamma^k r_{\lambda}(x_k,i_k,a_k)\mid a_k\sim\mu, P^{\text{aug}}_0\right)\right)\\
= &\sum_{x^0,i^0}P^{\text{aug}}_0(x^0,i^0)\sum_{(x',i')\in\bar{\mathcal X}}\left(d_{P^{\text{aug}},\mu}(x',i')-d_{{\widehat{P}}^{\text{aug}},\mu}(x',i')\right)\sum_{a'\in\bar{\mathcal A}}\mu(a'|x',i') r_\lambda(x',i',a')\\
=&\sum_{x^0,i^0}P^{\text{aug}}_0(x^0,i^0)  r_{\lambda,\mu}^T\left(I-\gamma  P^{\text{aug}}_{\mu}\right)^{-\top}\gamma \left((P^{\text{aug}}_\mu)^\top -({\widehat{P}}^{\text{aug}}_\mu)^\top \right)d_{{\widehat{P}}^{\text{aug}},\mu}
\end{split}
\]
where $ r_{\lambda,\mu}(x,i)= r_\lambda(x,i,\mu(x,i))$ is the reward functions induced by policy $\mu$. Furthermore, the following expressions can be easily obtained by using the properties of dual MDPs:
\[
\begin{split}
\left\{\left(I-\gamma  P^{\text{aug}}_{\mu}\right)^{-1} r_{\lambda,\mu}\right\}(x,i)=&\sum_{(x',i')\in\bar{\mathcal X}}\sum_{k=0}^\infty \gamma^k\mathbb P_{  P^{\text{aug}}}(x_{k}=x',i_{k}=i'|x_0=x,i_0=i,\mu) r_{\lambda,\mu}(x',i'),\\
=&\mathbb E_{ P^{\text{aug}}}\left[\sum_{k=0}^\infty\gamma^k  r_{\lambda,\mu}(x_k,i_k)|x_0=x,i_0=i,\mu\right],
\end{split}
\]
and
\[
\begin{split}
&\left\{\left((P^{\text{aug}}_\mu)^\top -({\widehat{P}}^{\text{aug}}_\mu)^\top \right)d_{{\widehat{P}}^{\text{aug}},\mu}\right\}(x',i')\\
=&\sum_{(x,i)\in\bar{\mathcal X}}\sum_{k=0}^\infty \gamma^k(1-\gamma)\mathbb P_{{\widehat{P}}^{\text{aug}}}(x_k=x,i_k=i|x_0=x^0,i_0=i^0,\mu)\sum_{a\in\bar{\mathcal A}}\left( P^{\text{aug}}(x',i'|x,i,a)-{\widehat{P}}^{\text{aug}}(x',i'|x,i,a)\right)\mu(a|x,i)\\
=&\mathbb E_{{\widehat{P}}^{\text{aug}}}\left[\sum_{k=0}^\infty \gamma^k(1-\gamma)\left( P^{\text{aug}}_{\mu}(x',i'|x_k,i_k)-{\widehat{P}}^{\text{aug}}_{\mu}(x',i'|x_k,i_k)\right)|x_0=x^0,i_0=i^0,\mu\right].
\end{split}
\]
By using the above results, the following expression holds:
\begin{equation}\label{eq:suff_ineq}
\begin{split}
&\mathbb E_{  P^{\text{aug}}}\left(\sum_{k=0}^\infty\gamma^k r_{\lambda}(x_k,i_k,a_k)\mid a_k\sim\mu, P^{\text{aug}}_0\right)-\mathbb E_{{\widehat{P}}^{\text{aug}}}\left(\sum_{k=0}^\infty\gamma^k r_{\lambda}(x_k,i_k,a_k)\mid a_k\sim\mu, P^{\text{aug}}_0\right)\\
= &\sum_{x^0,i^0}P^{\text{aug}}_0(x^0,i^0)\Bigg\{ \gamma\sum_{(x,i)\in\bar{\mathcal X}}\mathbb E_{ P^{\text{aug}}}\left[\sum_{k=0}^\infty\gamma^k  r_{\lambda,\mu}(x_k,i_k)|x_0=x,i_0=i,\mu\right]\cdot\\
&\quad\quad\quad \mathbb E_{{\widehat{P}}^{\text{aug}}}\left[\sum_{k=0}^\infty \gamma^k\left(  P^{\text{aug}}_{\mu}(x,i|x_k,i_k)-{\widehat{P}}^{\text{aug}}_{\mu}(x,i|x_k,i_k)\right)|x_0=x^0,i_0=i^0,\mu\right]\\
\geq & -\frac{  2\gamma\lambda R_{\max}}{1-\gamma}\mathbb E_{{\widehat{P}}^{\text{aug}}}\left[\sum_{k=0}^\infty \gamma^k e(x_k,a_k)\mathbf 1\{i_k=1\}|a_k\sim\mu, P^{\text{aug}}_0\right]
\end{split}
\end{equation}
where $2\lambda R_{\max}$ is the upper bound for reward function $r_{\lambda}(x,i,a)$ at $i=1$ with $\lambda\geq 0$.

Thus, by noticing that
\begin{equation}\label{eq:diff_1}
\begin{split}
&\mathbb E_{{\widehat{P}}^{\text{aug}}}\left(\sum_{k=0}^\infty\gamma^k r_{\lambda}(x_k,i_k,a_k)\mid a_k\sim\mu, P^{\text{aug}}_0\right)-\frac{ 2\gamma\lambda R_{\max}}{1-\gamma}\mathbb E_{{\widehat{P}}^{\text{aug}}}\left[\sum_{k=0}^\infty \gamma^k e(x_k,a_k)\mathbf 1\{i_k=1\}|a_k\sim\mu, P^{\text{aug}}_0\right]\\
=&\mathbb E_{{\widehat{P}}^{\text{aug}}}\left(\sum_{k=0}^\infty\gamma^k  \widehat{r}_{\lambda}(x_k,i_k,a_k)\mid a_k\sim\mu, P^{\text{aug}}_0\right),
\end{split}
\end{equation}
and recalling the definition of $\widehat{L}(\mu,\lambda)$, one concludes that $L(\mu,\lambda)-\widehat{L}(\mu,\lambda)\geq 0$. 
\end{proof}
By solving for $\max_{\mu}\min_{\lambda\geq 0} \widehat{L}(\mu,\lambda)$ and guaranteeing that it is lower bounded, one can ensure feasibility of the risk neutral maximization problem. 

\section{Algorithm}\label{sec:algorithm}
We first state the following algorithm for finding the saddle point $(\widehat{\mu},\widehat{\lambda})$ of the max-min optimization problem $\max_{\mu}\min_{\lambda\geq 0} \widehat{L}(\mu,\lambda)$.
\begin{algorithm}
\begin{algorithmic}
\STATE {\bf Input:} threshold $M_B$
\STATE {\bf Initialization:} policy parameter $\mu=\mu^{(0)}$, Lagrangian parameter $\lambda=\lambda^{(0)}$ and initial dual function estimate $f_{\text{min}}^{(0)}$
\FOR{$j = 0,1,2,\ldots$}
\STATE $q\leftarrow 0$.
\WHILE{TRUE}
\STATE Policy evaluation: For given policy $\mu^{(q)}$, compute the value function $\widehat{V}^{\mu^{(q)}}_{\lambda^{(j)}}$ by solving the following optimization problem,
\begin{alignat*}{2}
\min_{V}& &\quad &\sum_{x,i} P^{\text{aug}}_0(x,i) V(x,i)\label{problem_phi_3}\\
&\text{s.t.} &\quad & \widehat{r}_{\lambda^{(j)}}(x,i,\mu^{(q)}(x,i))+\gamma\sum_{(x',i')\in\bar{\mathcal X}} \widehat{P}(x'|x,\mu^{(q)}(x,i))V(x',i')\mathbf 1\{i=i'\}\leq V(x,i),\,\,\forall (x,i)\in\bar{\mathcal X}.\nonumber
\end{alignat*}
\STATE Policy improvement:
\begin{equation*}
\mu^{(q+1)}(x,i)\!\in\!\arg\!\max_{a\in \mathcal A}\left\{\! \widehat{r}_{\lambda^{(j)}}(x,i,a)\!+
\sum_{(x',i')\in\bar{\mathcal X}} \widehat{P}(x'|x,a)\widehat{V}^{\mu^{(q)}}_{\lambda^{(j)}}(x',i')\mathbf 1\{i=i'\}\leq \widehat{V}^{\mu^{(q)}}_{\lambda^{(j)}}(x,i)\right\}.\!\!
\end{equation*}
\STATE $q\leftarrow q+1$.
\IF{$\widehat{V}^{\mu^{(q)}}_{\lambda^{(j)}}(x,i)=  \widehat{V}^{\mu^{(q+1)}}_{\lambda^{(j)}}(x,i),\,\,\forall (x,i)\in\bar{\mathcal X}$}
\STATE{{\bf return} $\widehat{\mu}_{\lambda^{(j)}}=\mu^{(q)}$ and {\bf break}.}
\ENDIF
\ENDWHILE  
Lagrange multiplier update:
\begin{equation*}
\lambda^{(j+1)}=\left(\lambda^{(j)}-\alpha^{(j)} \left(
\mathbb E_{\widehat{P}}\left(\sum_{k=0}^\infty\gamma^k  \left( r(x_k,a_k)-\frac{  \gamma R_{\max}}{1-\gamma}e(x_k,a_k)\right)\mid a_k\sim\widehat{\mu}_{\lambda^{(j)}},P_0\right)- M_B
\right)\right)^+
\end{equation*}
 where the step length $\alpha^{(j)}$ satisfies the following conditions:
\begin{equation}\label{step_length_cond}
\alpha^{(j)}\geq 0,\,\,\sum_{j=0}^\infty\alpha^{(j)}=\infty,\,\,\sum_{j=0}^\infty\left(\alpha^{(j)}\right)^2<\infty.
\end{equation}
\STATE Best dual function estimate update: $f_{\text{min}}^{(j+1)}=\min\left(f_{\text{min}}^{(j)},\widehat{L}(\widehat{\mu}_{\lambda^{(j+1)}},\lambda^{(j+1)})\right)$.
\STATE Best Lagrange multiplier update: 
\begin{small}
\[
\lambda^{(j+1)}\leftarrow\left\{\begin{array}{ll}
\lambda^{(j+1)}&\text{if  $f_{\text{min}}^{(j+1)}=\widehat{L}(\widehat{\mu}_{\lambda^{(j+1)}},\lambda^{(j+1)})$}\\
\lambda^{(j)}&\text{otherwise}
\end{array}\right..
\]
\end{small}
\STATE Terminate algorithm when $\lambda^{(j)}$ converges to $\widehat{\lambda}$.
\ENDFOR
\end{algorithmic}
\end{algorithm}
The convergence analysis of this algorithm is given in the following theorem. 
\begin{theorem}[Convergence]\label{thm:convergence}
The policy iteration algorithm terminates in a finite number of steps. Furthermore, 
the Markov stationary control policy $\widehat{\mu}:=\widehat{\mu}_{\widehat{\lambda}}$ is optimal to the risk neutral optimization problem 
\begin{alignat*}{2}
\max_{\mu\in\Pi_S}& &\quad &\mathbb E_{\widehat{P}}\left(\sum_{k=0}^\infty\gamma^k r(x_k,a_k)\mid a_k\sim\mu,P_0\right)\\  
\text{subject to} & & \quad & \mathbb E_{\widehat{P}}\left(\sum_{k=0}^\infty\gamma^k r(x_k,a_k)\mid a_k\sim\mu,P_0\right)\geq M_B,
\end{alignat*}
when this algorithm terminates and
\[
\sum_{x,i} P^{\text{aug}}_0(x,i) \widehat{V}^{\widehat{\mu}}_{\widehat{\lambda}}(x,i)-\widehat{\lambda}M_B=\widehat{\mathcal P}(P^{\text{aug}}_0,M_B),
\]
where
\[
\widehat{\mathcal P}(P^{\text{aug}}_0,M_B)=\max_{\mu}\min_{\lambda\geq 0}\,\mathbb E_{  \widehat{P}^{\text{aug}}}\left(\sum_{k=0}^\infty\gamma^k \widehat{r}_{\lambda}(x_k,i_k,a_k)\mid a_k\sim\mu, P^{\text{aug}}_0\right)-\lambda M_B=\max_{\mu}\min_{\lambda\geq 0} \widehat{L}(\mu,\lambda).
\]
\end{theorem}
\begin{proof}
Define the Bellman operators with respect to $\widehat{P}^{\text{aug}}$ and $\widehat{r}_\lambda$ as follows:
\[
\widehat{T}_{\lambda}[V](x,i)=\max_{a\in\bar{\mathcal A}}\left\{\widehat{r}_{\lambda}(x,i,a)+\gamma\sum_{(x^\prime, i^\prime)\in\bar{\mathcal X}}\widehat{P}^{\text{aug}}(x^\prime,i^\prime|x,i,a)V(x^\prime,i^\prime)\right\},
\]
and 
\[
\widehat{T}^\mu_{\lambda}[V](x,i)=\widehat{r}_{\lambda}(x,i,\mu(x,i))+\gamma\sum_{(x^\prime, i^\prime)\in\bar{\mathcal X}}\widehat{P}^{\text{aug}}(x^\prime,i^\prime|x,i,\mu(x,i))V(x^\prime,i^\prime).
\]
From the policy update rule in the above algorithm, one obtains
\[
\widehat{T}^{\mu^{(q+1)}}_{\lambda^{(j)}}\left[\widehat{V}^{\mu^{(q)}}_{\lambda^{(j)}}\right](x,i)=\widehat{T}_{\lambda^{(j)}}\left[\widehat{V}^{\mu^{(q)}}_{\lambda^{(j)}}\right](x,i),\,\,\forall (x,i)\in\bar{\mathcal X}.
\]
On the other hand, by the fixed point equation,
\[
\widehat{T}^{\mu^{(q)}}_{\lambda^{(j)}}\left[\widehat{V}^{\mu^{(q)}}_{\lambda^{(j)}}\right](x,i)=\widehat{V}^{\mu^{(q)}}_{\lambda^{(j)}}(x,i)
\]
and notice that $\widehat{T}^{\mu^{(q)}}_{\lambda^{(j)}}\left[V\right](x,i)\leq \widehat{T}_{\lambda^{(j)}}[V](x,i)$, one obtains
\[
\widehat{V}^{\mu^{(q)}}_{\lambda^{(j)}}(x,i)\leq \widehat{T}^{\mu^{(q+1)}}_{\lambda^{(j)}}\left[\widehat{V}^{\mu^{(q)}}_{\lambda^{(j)}}\right](x,i).
\]
By repeatedly applying $\widehat{T}^{\mu^{(q+1)}}_{\lambda^{(j)}}$ on both sides, this inequality also implies
\begin{equation}\label{eq:conv_1}
\widehat{V}^{\mu^{(q)}}_{\lambda^{(j)}}(x,i)\leq \widehat{T}^{\mu^{(q+1)}}_{\lambda^{(j)}}\left[\widehat{V}^{\mu^{(q)}}_{\lambda^{(j)}}\right](x,i)\leq\cdots\leq \lim_{n\rightarrow\infty}\left(\widehat{T}^{\mu^{(q)}}_{\lambda^{(j)}}\right)^n\left[\widehat{V}^{\mu^{(q)}}_{\lambda^{(j)}}\right](x,i)\leq \widehat{V}^{\mu^{(q+1)}}_{\lambda^{(j)}}(x,i)
\end{equation}
for any $(x,i)\in\bar{\mathcal X}$. When the inner-loop stopping condition is satisfied: $\widehat{V}^{\mu^{(q+1)}}_{\lambda^{(j)}}(x,i)=\widehat{V}^{\mu^{(q)}}_{\lambda^{(j)}}(x,i)$, then
\[
\widehat{T}_{\lambda^{(j)}}\left[\widehat{V}^{\mu^{(q)}}_{\lambda^{(j)}}\right](x,i)=\widehat{T}^{\mu^{(q+1)}}_{\lambda^{(j)}}\left[\widehat{V}^{\mu^{(q)}}_{\lambda^{(j)}}\right](x,i)\\
=\widehat{T}^{\mu^{(q+1)}}_{\lambda^{(j)}}\left[\widehat{V}^{\mu^{(q+1)}}_{\lambda^{(j)}}\right](x)=\widehat{V}^{\mu^{(q+1)}}_{\lambda^{(j)}}(x,i)=\widehat{V}^{\mu^{(q)}}_{\lambda^{(j)}}(x,i)
\]
The first equality is based on the definition of Bellman operator $\widehat{T}_\lambda$ and the policy iteration algorithm. The third equality is due to the fact that $\widehat{V}^{\mu^{(q+1)}}_{\lambda^{(j)}}$ is a solution to fixed point equation: $\widehat{T}^{\mu^{(q+1)}}_{\lambda^{(j)}}[V](x,i)=V(x,i)$. The second and the fourth equality are due to the inner-loop stopping condition. 
Notice that the fixed point of $\widehat{T}_{\lambda^{(j)}}[V](x,i)=V(x,i)$ is unique. Combining these arguments, one obtains $\widehat{V}^{\mu^{(q)}}_{\lambda^{(j)}}(x,i)=\widehat{V}_{\lambda^{(j)}}(x,i)$ where $\widehat{V}_{\lambda^{(j)}}(x,i)=\max_{\mu}\,\mathbb E_{  \widehat{P}^{\text{aug}}}\left(\sum_{k=0}^\infty\gamma^k \widehat{r}_{\lambda^{(j)}}(x_k,i_k,a_k)\mid a_k\sim\mu,x_0=x,i_0=i\right)$. This implies $\widehat{\mu}_{\lambda^{(j)}}=\mu^{(q)}\in\arg\max_\mu \widehat{L}(\mu,\lambda^{(j)})$.

Second we want to show that the sequence of $\lambda^{(j)}$ converges to the global minimum $\widehat{\lambda}\in\argmin_{\lambda\geq 0}f(\lambda)$ where $f(\lambda)=\max_\mu \widehat{L}(\mu,\lambda)$. Notice by definition that $\widehat{L}(\mu,\lambda)$ is a linear function of $\lambda$. Furthermore, since the policy $\mu$ belongs to the closed subset of the compact set of simplexes, then $\max_\mu \widehat{L}(\mu,\lambda)$ is the point-wise supremum over an infinite set of convex functions. This implies that $f(\lambda)=\max_\mu \widehat{L}(\mu,\lambda)$ is convex in $\lambda$. It can be easily shown by sub-gradient calculus that
\[
\frac{d \widehat{L}(\mu,\lambda)}{d\lambda}\bigg\vert_{\mu=\mu^\ast_\lambda}=\mathbb E_{\widehat{P}}\left(\sum_{k=0}^\infty\gamma^k  \left(  r(x_k,a_k)-\frac{  \gamma R_{\max}}{1-\gamma}e(x_k,a_k)\right)\mid a_k\sim\widehat{\mu}_{\lambda},P_0\right)- M_B
\in\partial f(\lambda).
\]
This implies that the Lagrange multiplier update can be viewed as projected sub-gradient descent. Now, for 
\[
\bar\lambda^{(j+1)}:=\lambda^{(j)}-\alpha^{(j)}\frac{d \widehat{L}(\mu,\lambda)}{d\lambda}\bigg\vert_{\mu=\widehat{\mu}_\lambda,\lambda=\lambda^{(j)}}
\]
and $\lambda^{(j+1)}=\left(\bar\lambda^{(j+1)}\right)^+$, since the projection operator is non-expansive, one obtains $(\lambda^{(j+1)}-\widehat{\lambda})^2\leq(\bar\lambda^{(j+1)}-\widehat{\lambda})^2$. 

Therefore, the following expression holds:
\[
\begin{split}
(\lambda^{(j+1)}-\widehat{\lambda})^2\leq &(\bar\lambda^{(j+1)}-\widehat{\lambda})^2\\
=&\left(\lambda^{(j)}-\alpha^{(j)}\frac{d \widehat{L}(\mu,\lambda)}{d\lambda}\bigg\vert_{\mu=\widehat{\mu}_\lambda,\lambda=\lambda^{(j)}}-\widehat{\lambda}\right)^2\\
=&(\lambda^{(j)}-\widehat{\lambda})^2-2\alpha^{(j)}(\lambda^{(j)}-\widehat{\lambda})\frac{d \widehat{L}(\mu,\lambda)}{d\lambda}\bigg\vert_{\mu=\widehat{\mu}_\lambda,\lambda=\lambda^{(j)}}+\left(\alpha^{(j)}\right)^2\left(\frac{d \widehat{L}(\mu,\lambda)}{d\lambda}\bigg\vert_{\mu=\widehat{\mu}_\lambda,\lambda=\lambda^{(j)}}\right)^2\\
\leq &(\lambda^{(j)}-\widehat{\lambda})^2-2\alpha^{(j)}(f(\lambda^{(j)})-f(\widehat{\lambda}))+\left(\alpha^{(j)}\right)^2\left(\frac{d \widehat{L}(\mu,\lambda)}{d\lambda}\bigg\vert_{\mu=\widehat{\mu}_\lambda,\lambda=\lambda^{(j)}}\right)^2,
\end{split}
\]
The inequality is due to the fact that for any $\lambda\geq 0$ and ${d \widehat{L}(\mu,\lambda)}/{d\lambda}\vert_{\mu=\widehat{\mu}_\lambda}\in\partial f(\lambda)$,
\[
f(\lambda)-f(\widehat{\lambda})\geq (\lambda-\widehat{\lambda})\frac{d \widehat{L}(\mu,\lambda)}{d\lambda}\bigg\vert_{\mu=\widehat{\mu}_\lambda}.
\]
This further implies
\[
(\lambda^{(j+1)}-\widehat{\lambda})^2\leq (\lambda^{(0)}-\widehat{\lambda})^2-\sum_{q=0}^j2\alpha^{(q)}(f(\lambda^{(q)})-f(\widehat{\lambda}))+\left(\alpha^{(q)}\right)^2\left(\frac{d \widehat{L}(\mu,\lambda)}{d\lambda}\bigg\vert_{\mu=\widehat{\mu}_\lambda,\lambda=\lambda^{(q)}}\right)^2.
\]
Since $(\lambda^{(j+1)}-\widehat{\lambda})^2$ is a positive quantity and $(\lambda^{(0)}-\widehat{\lambda})^2$ is bounded, this further implies
\[
2\sum_{q=0}^j\alpha^{(q)}(f(\lambda^{(q)})-f(\widehat{\lambda}))\leq (\bar\lambda^{(0)}-\widehat{\lambda})^2+\sum_{q=0}^j\left(\alpha^{(q)}\right)^2\left(\frac{d \widehat{L}(\mu,\lambda)}{d\lambda}\bigg\vert_{\mu=\widehat{\mu}_\lambda,\lambda=\lambda^{(q)}}\right)^2.
\]
By defining $f_{\text{min}}^{(j)}=\min_{q\in\{0,\ldots,j\}}f(\lambda^{(q)})$, the above expression implies
\[
f_{\text{min}}^{(j)}-f(\widehat{\lambda})\leq \frac{1}{\sum_{q=0}^j\alpha^{(q)}}\left((\lambda^{(0)}-\lambda^*)^2+\sum_{q=0}^j\left(\alpha^{(q)}\right)^2\left(\frac{d \widehat{L}(\mu,\lambda)}{d\lambda}\bigg\vert_{\mu=\widehat{\mu}_\lambda,\lambda=\lambda^{(q)}}\right)^2\right).
\]
The step-size rule in \eqref{step_length_cond} for $\alpha^{(q)}$ ensures that the numerator is bounded and the denominator goes to infinity as $j\rightarrow\infty$.
This implies that for any $\epsilon>0$, there exists a constant $N(\epsilon)$ such that for any $j>N(\epsilon)$, $f(\widehat{\lambda}) \leq f_{\text{min}}^{(j)}\leq f(\widehat{\lambda})+\epsilon$.
In other words, the sequence $\lambda^{(j)}$ converges to the global minimum $\widehat{\lambda}$ of $f(\lambda)$ where $f(\lambda)=\max_{\mu} \widehat{L}(\mu,\lambda)$. 

By combining all previous arguments, one shows that $\widehat{L}(\widehat{\mu},\widehat{\lambda})=\widehat{\mathcal D}(P^{\text{aug}}_0,M_B)$, where
\[
\widehat{\mathcal D}(P^{\text{aug}}_0,M_B)=\min_{\lambda\geq 0}\max_{\mu}\,\mathbb E_{  \widehat{P}^{\text{aug}}}\left(\sum_{k=0}^\infty\gamma^k \widehat{r}_{\lambda}(x_k,i_k,a_k)\mid a_k\sim\mu, P^{\text{aug}}_0\right)-\lambda M_B=\min_{\lambda\geq 0}\max_{\mu} \widehat{L}(\mu,\lambda).
\]
Then the convergence proof is completed by using strong duality, i.e., $\widehat{\mathcal D}(P^{\text{aug}}_0,M_B)=\widehat{\mathcal P}(P^{\text{aug}}_0,M_B)$. 
\end{proof}
In the next section, we will investigate the performance in terms of sub-optimality for the solution of this algorithm.
  
\section{Performance}\label{sec:perf}
We are now in position to derive a performance bound on the control policies found by the above algorithm. By the primal formulation of Lagrangian function, one obtains the following expressions:
\begin{alignat*}{2}
\max_{\mu}\min_{\lambda\geq 0} {L}(\mu,\lambda)=&\max_{\mu\in\Pi_S} &\quad &\mathbb E_{\widehat{P}}\left(\sum_{k=0}^\infty\gamma^k r(x_k,a_k)\mid a_k\sim\mu,P_0\right)\\  
&\text{subject to}  & \quad & \mathbb E_{P}\left(\sum_{k=0}^\infty\gamma^k r(x_k,a_k)\mid a_k\sim\mu,P_0\right)\geq M_B,
\end{alignat*}
and  $\min_{\lambda\geq 0} {L}(\widehat{\mu},\lambda)=\mathbb E_{\widehat{P}}\left(\sum_{k=0}^\infty\gamma^k r(x_k,a_k)\mid a_k\sim\widehat{\mu},P_0\right)$ such that $\mathbb E_{P}\left(\sum_{k=0}^\infty\gamma^k r(x_k,a_k)\mid a_k\sim\widehat{\mu},P_0\right)\geq M_B$. Let $\mu^*$ be the optimal policy of the above constrained problem. In order to compare the performance of the policy from approximation $\widehat{\mu}$, one needs to calculate the upper/lower bound for the following expression: 
\[
\mathbb E_{{P}}\left(\sum_{k=0}^\infty\gamma^k r(x_k,a_k)\mid a_k\sim\mu^*,P_0\right)-\mathbb E_{{P}}\left(\sum_{k=0}^\infty\gamma^k r(x_k,a_k)\mid a_k\sim\widehat{\mu},P_0\right).
\]
By the primal formulation of Lagrangian function, since both $\widehat{\mu}$ and $\mu^*$ are feasible control policies and the stage-wise reward functions are bounded, one can easily check that the following quantities are bounded: $\min_{\lambda\geq 0} {L}(\widehat{\mu},\lambda)$ and $\max_{\mu}\min_{\lambda\geq 0} {L}(\mu,\lambda)$. Furthermore, this also implies that
$\mathbb E_{{P}}\left(\sum_{k=0}^\infty\gamma^k r(x_k,a_k)\mid a_k\sim\mu^*,P_0\right)-\mathbb E_{{P}}\left(\sum_{k=0}^\infty\gamma^k r(x_k,a_k)\mid a_k\sim\widehat{\mu},P_0\right)=\max_{\mu}\min_{\lambda\geq 0} {L}(\mu,\lambda)-\min_{\lambda\geq 0} {L}(\widehat{\mu},\lambda)$. 

Define the Bellman residual as follows:
\[
\text{BR}(\widehat{\mu},\widehat{\lambda})=\left\|T_{\widehat{\lambda}} [{V}^{\widehat{\mu}}_{\widehat{\lambda}}]-{V}^{\widehat{\mu}}_{\widehat{\lambda}}\right\|_{\infty}.
\]
Based on the definitions of stage-wise reward functions $r_{\widehat{\lambda}}$ and $\widehat{r}_{\widehat{\lambda}}$, when the probability mis-match error function $e(x,i,a)$ tends to zero for any $(x,i,a)$, the fixed point solution calculated by the approximated policy ${V}^{\widehat{\mu}}_{\widehat{\lambda}}(x,i)$ converges to the fixed point solution induced by the optimal policy ${V}^{\tilde\mu}_{\widehat{\lambda}}(x,i)$ for any $(x,i)$. Notice that $T_{\widehat{\lambda}} [{V}^{\tilde\mu}_{\widehat{\lambda}}](x,i)={V}^{\tilde\mu}_{\widehat{\lambda}}(x,i)$ for any $(x,i)$. Therefore, when $\max_{x,i,a}e(x,i,a)\rightarrow 0$, one obtains ${V}^{\widehat{\mu}}_{\widehat{\lambda}}(x,i)\rightarrow {V}^{\tilde\mu}_{\widehat{\lambda}}(x,i)$, $T_{\widehat{\lambda}} [{V}^{\widehat{\mu}}_{\widehat{\lambda}}](x,i)\rightarrow T_{\widehat{\lambda}} [{V}^{\tilde\mu}_{\widehat{\lambda}}](x,i)={V}^{\tilde\mu}_{\widehat{\lambda}}(x,i)$ for any $(x,i)$ and the Bellman residual $\text{BR}(\widehat{\mu},\widehat{\lambda})$ tends to zero. 

First, we introduce the following Lemma on a sub-optimality performance guarantee.
\begin{lemma}[Performance]\label{lem:perf}
Let $(\widehat{\mu},\widehat{\lambda})$ be the saddle point solution of $\min_{\lambda\geq 0}\max_{\mu} \widehat{L}(\mu,\lambda)$.  Then the following performance bound holds:
\[
\begin{split}
0\leq & \mathbb E_{{P}}\left(\sum_{k=0}^\infty\gamma^k r(x_k,a_k)\mid a_k\sim\mu^*,P_0\right)-\mathbb E_{{P}}\left(\sum_{k=0}^\infty\gamma^k r(x_k,a_k)\mid a_k\sim\widehat{\mu},P_0\right)\\
\leq & \frac{ 2\gamma\widehat{\lambda} R_{\max}}{1-\gamma}\mathbb E_{\widehat{P}}\left[\sum_{k=0}^\infty \gamma^k e(x_k,a_k)|a_k\sim\widehat{\mu}, P_0\right]+\frac{\text{BR}(\widehat{\mu},\widehat{\lambda})}{1-\gamma}.
\end{split}
\]
\end{lemma}
\begin{proof}
Recall that $\widehat{L}(\mu,\lambda)\leq {L}(\mu,\lambda)$ for any $\mu$ and $\lambda\geq 0$. Therefore, by defining $\tilde\mu\in\arg\max_{\mu} {L}(\mu,\widehat{\lambda})$,
the following expression holds:
\begin{equation}\label{eq:bdd}
\begin{split}
0\leq & \min_{\lambda\geq 0}\max_{\mu} {L}(\mu,\lambda)-\min_{\lambda\geq 0}\max_{\mu} \widehat{L}(\mu,\lambda)\leq  \max_{\mu} {L}(\mu,\widehat{\lambda})-\widehat{L}(\widehat{\mu},\widehat{\lambda})= L(\tilde\mu,\widehat{\lambda})-\widehat{L}(\widehat{\mu},\widehat{\lambda})\\
= &\sum_{x,i} P^{\text{aug}}_0(x,i)\left({V}^{\tilde\mu}_{\widehat{\lambda}}(x,i)  -\widehat{V}^{\widehat{\mu}}_{\widehat{\lambda}}(x,i)\right)\\
= &\sum_{x,i} P^{\text{aug}}_0(x,i)\left({V}^{\tilde\mu}_{\widehat{\lambda}}(x,i)-{V}^{\widehat{\mu}}_{\widehat{\lambda}}(x,i)+{V}^{\widehat{\mu}}_{\widehat{\lambda}}(x,i)-\widehat{V}^{\widehat{\mu}}_{\widehat{\lambda}}(x,i)  \right)
\end{split}
\end{equation}
The first inequality is due to the fact that $\widehat{L}(\mu,\lambda)\leq {L}(\mu,\lambda)$. The second inequality is based on the definition of the saddle point solution of $\min_{\lambda\geq 0}\max_{\mu} \widehat{L}(\mu,\lambda)$, and the fact that $\widehat{\lambda}$ is a feasible solution to $\min_{\lambda\geq 0}\max_{\mu} {L}(\mu,\lambda)$. The first equality follows from definition and the second equality follows from the result of Bellman's equation. The last equality is a result of basic algebraic manipulations. 

In order to derive a performance bound between the approximated solution and the optimal solution. We study the rightmost term $\sum_{x,i} P^{\text{aug}}_0(x,i)\left({V}^{\tilde\mu}_{\widehat{\lambda}}(x,i)-{V}^{\widehat{\mu}}_{\widehat{\lambda}}(x,i)+{V}^{\widehat{\mu}}_{\widehat{\lambda}}(x,i)-\widehat{V}^{\widehat{\mu}}_{\widehat{\lambda}}(x,i)  \right)$ of expression \eqref{eq:bdd}.
Consider the first part of this expression: 
\[
\sum_{x,i} P^{\text{aug}}_0(x,i)\left({V}^{\widehat{\mu}}_{\widehat{\lambda}}(x,i)-\widehat{V}^{\widehat{\mu}}_{\widehat{\lambda}}(x,i)\right).
\]
Based on previous analysis in Bellman's recursion, one can easily see that
\[
\begin{split}
\sum_{x,i} P^{\text{aug}}_0(x,i)&\left({V}^{\widehat{\mu}}_{\widehat{\lambda}}(x,i)-\widehat{V}^{\widehat{\mu}}_{\widehat{\lambda}}(x,i)\right)=\mathbb E_{  \widehat{P}^{\text{aug}}}\left(\sum_{k=0}^\infty\gamma^k \left({r}_{\widehat{\lambda}}(x_k,i_k,a_k)- \widehat{r}_{\widehat{\lambda}}(x_k,i_k,a_k)\right)\mid a_k\sim\widehat{\mu}, P^{\text{aug}}_0\right)\\
&+\mathbb E_{  {P}^{\text{aug}}}\left(\sum_{k=0}^\infty\gamma^k  r_{\widehat{\lambda}}(x_k,i_k,a_k)\mid a_k\sim\widehat{\mu}, P^{\text{aug}}_0\right)-\mathbb E_{  \widehat{P}^{\text{aug}}}\left(\sum_{k=0}^\infty\gamma^k  r_{\widehat{\lambda}}(x_k,i_k,a_k)\mid a_k\sim\widehat{\mu}, P^{\text{aug}}_0\right).
\end{split}
\]
By expression \eqref{eq:diff_1}, one notices that
\[
\begin{split}
&\mathbb E_{  \widehat{P}^{\text{aug}}}\left(\sum_{k=0}^\infty\gamma^k ({r}_{\widehat{\lambda}}(x_k,i_k,a_k)- \widehat{r}_{\widehat{\lambda}}(x_k,i_k,a_k))\mid a_k\sim\widehat{\mu}, P^{\text{aug}}_0\right)\\
=&\frac{ 2\gamma\widehat{\lambda} R_{\max}}{1-\gamma}\mathbb E_{{\widehat{P}}^{\text{aug}}}\left[\sum_{k=0}^\infty \gamma^k e(x_k,a_k)\mathbf 1\{i_k=1\}|a_k\sim\widehat{\mu}, P^{\text{aug}}_0\right]=\frac{ \gamma\widehat{\lambda} R_{\max}}{1-\gamma}\mathbb E_{{\widehat{P}}}\left[\sum_{k=0}^\infty \gamma^k e(x_k,a_k)|a_k\sim\widehat{\mu}, P_0\right].
\end{split}
\]
On the other hand, similar to the analysis in expression \eqref{eq:suff_ineq}, one obtains
\begin{equation}
\begin{split}
&\mathbb E_{  P^{\text{aug}}}\left(\sum_{k=0}^\infty\gamma^k r_{\widehat{\lambda}}(x_k,i_k,a_k)\mid a_k\sim\widehat{\mu}, P^{\text{aug}}_0\right)-\mathbb E_{{\widehat{P}}^{\text{aug}}}\left(\sum_{k=0}^\infty\gamma^k r_{\widehat{\lambda}}(x_k,i_k,a_k)\mid a_k\sim\widehat{\mu}, P^{\text{aug}}_0\right)\\
= &\sum_{x^0,i^0}P^{\text{aug}}_0(x^0,i^0)\Bigg\{ \gamma\sum_{(x,i)\in\bar{\mathcal X}}\mathbb E_{ P^{\text{aug}}}\left[\sum_{k=0}^\infty\gamma^k  r_{\widehat{\lambda},\mu}(x_k,i_k)|x_0=x,i_0=i,a_k\sim\widehat{\mu}\right]\cdot\\
&\quad\quad\quad \mathbb E_{{\widehat{P}}^{\text{aug}}}\left[\sum_{k=0}^\infty \gamma^k\left(  P^{\text{aug}}_{\mu}(x,i|x_k,i_k)-{\widehat{P}}^{\text{aug}}_{\mu}(x,i|x_k,i_k)\right)|x_0=x^0,i_0=i^0,a_k\sim\widehat{\mu}\right]\\
\leq & \frac{  2\gamma\widehat{\lambda} R_{\max}}{1-\gamma}\mathbb E_{{\widehat{P}}^{\text{aug}}}\left[\sum_{k=0}^\infty \gamma^k e(x_k,a_k)\mathbf 1\{i_k=1\}|a_k\sim\widehat{\mu}, P^{\text{aug}}_0\right]\\
=&\frac{ \gamma\widehat{\lambda} R_{\max}}{1-\gamma}\mathbb E_{{\widehat{P}}}\left[\sum_{k=0}^\infty \gamma^k e(x_k,a_k)|a_k\sim\widehat{\mu}, P_0\right].
\end{split}
\end{equation}
Therefore by combining these analysis,
\begin{equation}\label{eq:diff_V}
\sum_{x,i} P^{\text{aug}}_0(x,i)\left({V}^{\widehat{\mu}}_{\widehat{\lambda}}(x,i)-\widehat{V}^{\widehat{\mu}}_{\widehat{\lambda}}(x,i)\right)\leq \frac{ 2\gamma\widehat{\lambda} R_{\max}}{1-\gamma}\mathbb E_{{\widehat{P}}}\left[\sum_{k=0}^\infty \gamma^k e(x_k,a_k)|a_k\sim\widehat{\mu}, P_0\right].
\end{equation}

Next consider the second part in expression \eqref{eq:bdd}:
\[
\sum_{x,i} P^{\text{aug}}_0(x,i)\left({V}^{\tilde\mu}_{\widehat{\lambda}}(x,i)-{V}^{\widehat{\mu}}_{\widehat{\lambda}}(x,i)\right).
\]
Recall the Bellman residual: $\text{BR}(\widehat{\mu},\widehat{\lambda})=\left\|T_{\widehat{\lambda}} [{V}^{\widehat{\mu}}_{\widehat{\lambda}}]-{V}^{\widehat{\mu}}_{\widehat{\lambda}}\right\|_{\infty}$.
By applying the contraction mapping property on $T_{\widehat{\lambda}} [{V}^{\widehat{\mu}}_{\widehat{\lambda}}](x,i)-{V}^{\widehat{\mu}}_{\widehat{\lambda}}(x,i)$ for any $(x,i)$, one obtains
\[
T^2_{\widehat \lambda} [{V}^{\widehat{\mu}}_{\widehat{\lambda}}](x,i)-T_{\widehat{\lambda}}[{V}^{\widehat{\mu}}_{\widehat{\lambda}}](x,i)\leq \gamma\text{BR}(\widehat{\mu},\widehat{\lambda}).
\]
By an induction argument, the above expression becomes $T^N_{\widehat{\lambda}} [{V}^{\widehat{\mu}}_{\widehat{\lambda}}](x,i)-T^{N-1}_{\widehat{\lambda}}[{V}^{\widehat{\mu}}_{\widehat{\lambda}}](x,i)\leq \gamma^{N-1}\text{BR}(\widehat{\mu},\widehat{\lambda})$, for which by a telescoping sum, it further implies
\[
T^N_{\widehat{\lambda}} [{V}^{\widehat{\mu}}_{\widehat{\lambda}}](x,i)-{V}^{\widehat{\mu}}_{\widehat{\lambda}}(x,i)= \sum_{k=1}^{N}T^k_{\widehat{\lambda}} [{V}^{\widehat{\mu}}_{\widehat{\lambda}}](x,i)-T^{k-1}_{\widehat{\lambda}}[{V}^{\widehat{\mu}}_{\widehat{\lambda}}](x,i)
\leq  \sum_{k=1}^N\gamma^{k-1}\text{BR}(\widehat{\mu},\widehat{\lambda}).
\]
By letting $N\rightarrow\infty$ and noticing that $\lim_{N\rightarrow\infty}T^N_{\widehat{\lambda}} [{V}^{\widehat{\mu}}_{\widehat{\lambda}}](x,i)={V}^{\tilde\mu}_{\widehat{\lambda}}(x,i)$, one finally obtains
\[
\sum_{x,i} P^{\text{aug}}_0(x,i)\left({V}^{\tilde\mu}_{\widehat{\lambda}}(x,i)-{V}^{\widehat{\mu}}_{\widehat{\lambda}}(x,i)\right)\leq  \lim_{N\rightarrow\infty}\sum_{k=1}^N\gamma^{k-1}\text{BR}(\widehat{\mu},\widehat{\lambda})= \frac{\text{BR}(\widehat{\mu},\widehat{\lambda})}{1-\gamma}.
\]

By combining both parts in expression \eqref{eq:bdd}, the performance bound can be re-written as follows:
\begin{equation}\label{eq:perf_1}
\begin{split}
0\leq  \min_{\lambda\geq 0}\max_{\mu} {L}(\mu,\lambda)-\min_{\lambda\geq 0}\max_{\mu} \widehat{L}(\mu,\lambda)\leq & \frac{ 2\gamma\widehat{\lambda} R_{\max}}{1-\gamma}\mathbb E_{{\widehat{P}}}\left[\sum_{k=0}^\infty \gamma^k e(x_k,a_k)|a_k\sim\widehat{\mu}, P_0\right]\\
&+\frac{\text{BR}(\widehat{\mu},\widehat{\lambda})}{1-\gamma}.
\end{split}
\end{equation}
Furthermore, by defining $\tilde\lambda\in\arg\min_{\lambda\geq 0} {L}(\widehat{\mu},\lambda)$, one obtains
\[
\begin{split}
0\leq &\max_{\mu}\min_{\lambda\geq 0} {L}(\mu,\lambda)-{L}(\widehat{\mu},\tilde\lambda)\\
\leq & \min_{\lambda\geq 0}\max_{\mu} {L}(\mu,\lambda)-\widehat{L}(\widehat{\mu},\tilde\lambda)\\
\leq & \min_{\lambda\geq 0}\max_{\mu} {L}(\mu,\lambda)-\min_{\lambda\geq 0}\widehat{L}(\widehat{\mu},\lambda)\\
=&\min_{\lambda\geq 0}\max_{\mu} {L}(\mu,\lambda)-\max_{\mu}\min_{\lambda\geq 0}\widehat{L}(\mu,\lambda)\\
= &\min_{\lambda\geq 0}\max_{\mu} {L}(\mu,\lambda)-\min_{\lambda\geq 0}\max_{\mu}\widehat{L}(\mu,\lambda)\\
\leq &\frac{ 2\gamma\widehat{\lambda} R_{\max}}{1-\gamma}\mathbb E_{{\widehat{P}}^{\text{aug}}}\left[\sum_{k=0}^\infty \gamma^k e(x_k,a_k)\mathbf 1\{i_k=1\}|a_k\sim\widehat{\mu}, P_0\right]+\frac{\text{BR}(\widehat{\mu},\widehat{\lambda})}{1-\gamma}.
\end{split}
\]
The first inequality follows from the fact that ${L}(\widehat{\mu},\tilde\lambda)=\min_{\lambda\geq 0} {L}(\widehat{\mu},\lambda)\leq \max_{\mu}\min_{\lambda\geq 0} {L}(\mu,\lambda)$, where the first and second inequality in this expression follow from definition and the last equality follows from strong duality. The second inequality is based on the fact that ${L}(\mu,\lambda)\geq \widehat{L}(\mu,\lambda)$ and strong duality $\max_{\mu}\min_{\lambda\geq 0} {L}(\mu,\lambda)=\min_{\lambda\geq 0}\max_{\mu} {L}(\mu,\lambda)$. The third inequality and the first equality follows from definition. The second equality follows from strong duality: $\min_{\lambda\geq 0}\max_{\mu} \widehat{L}(\mu,\lambda)=\max_{\mu}\min_{\lambda\geq 0} \widehat{L}(\mu,\lambda)$ and the last inequality follows from the performance bound in \eqref{eq:perf_1}.
\end{proof}
Recall the Bellman residual: $\text{BR}(\widehat{\mu},\widehat{\lambda})=\left\|T_{\widehat{\lambda}} [{V}^{\widehat{\mu}}_{\widehat{\lambda}}]-{V}^{\widehat{\mu}}_{\widehat{\lambda}}\right\|_{\infty}$. In general, one cannot calculate the Bellman residual $\text{BR}(\widehat{\mu},\widehat{\lambda})$ because the transition probability $P^{\text{aug}}$ is not known. We therefore provide an upper bound for $\text{BR}(\widehat{\mu},\widehat{\lambda})$ in the next lemma.
\begin{lemma}\label{lem:tech}
Let $(\widehat{\mu},\widehat{\lambda})$ be the saddle point solution of $\min_{\lambda\geq 0}\max_{\mu} \widehat{L}(\mu,\lambda)$. 
Then the following expression holds:
\[
\text{BR}(\widehat{\mu},\widehat{\lambda})\leq \max_{x\in\mathcal X,a\in\mathcal A}\left\{\frac{ 2(1+\gamma)\gamma\widehat{\lambda} R_{\max}}{1-\gamma}\mathbb E_{{\widehat{P}}}\left[\sum_{k=0}^\infty \gamma^k e(x_k,a_k)|x_0=x,a_k\sim\widehat{\mu}\right]+\frac{  4\gamma\widehat{\lambda} R_{\max}}{1-\gamma}e(x,a)\right\}.
\]
\end{lemma}
\begin{proof}
By triangular inequality, one can easily see that
\[
\begin{split}
\text{BR}(\widehat{\mu},\widehat{\lambda})\leq &\left\|T_{\widehat{\lambda}} [\widehat{V}^{\widehat{\mu}}_{\widehat{\lambda}}]-\widehat{V}^{\widehat{\mu}}_{\widehat{\lambda}}\right\|_{\infty}+\left\|\widehat{V}^{\widehat{\mu}}_{\widehat{\lambda}}-{V}^{\widehat{\mu}}_{\widehat{\lambda}}\right\|_{\infty}+\left\|T_{\widehat{\lambda}} [{V}^{\widehat{\mu}}_{\widehat{\lambda}}]-T_{\widehat{\lambda}} [\widehat{V}^{\widehat{\mu}}_{\widehat{\lambda}}]\right\|_{\infty}\\
\leq &(1+\gamma)\left\|\widehat{V}^{\widehat{\mu}}_{\widehat{\lambda}}-{V}^{\widehat{\mu}}_{\widehat{\lambda}}\right\|_{\infty}+\left\|T_{\widehat{\lambda}} [\widehat{V}^{\widehat{\mu}}_{\widehat{\lambda}}]-\widehat{T}_{\widehat{\lambda}} [\widehat{V}^{\widehat{\mu}}_{\widehat{\lambda}}]\right\|_{\infty}
\end{split}
\]
The second inequality is due to contraction property of Bellman operator and the fixed point theorem: $\widehat{T}_{\widehat{\lambda}} [\widehat{V}^{\widehat{\mu}}_{\widehat{\lambda}}](x,i)=\widehat{V}^{\widehat{\mu}}_{\widehat{\lambda}}(x,i)$, for any $ (x,i)\in\bar{\mathcal X}$. Similar to the derivation in \eqref{eq:diff_V}, one obtains
\begin{equation*}
\left|{V}^{\widehat{\mu}}_{\widehat{\lambda}}(x,i)-\widehat{V}^{\widehat{\mu}}_{\widehat{\lambda}}(x,i)\right|\leq \frac{ 2\gamma\widehat{\lambda} R_{\max}}{1-\gamma}\mathbb E_{{\widehat{P}}}\left[\sum_{k=0}^\infty \gamma^k e(x_k,a_k)|x_0=x,a_k\sim\widehat{\mu}\right],\forall (x,i)\in\bar{\mathcal X}.
\end{equation*}
Furthermore for any $(x,i)\in\bar{\mathcal X}$,
\[
\begin{split}
&\left|T_{\widehat{\lambda}} [\widehat{V}^{\widehat{\mu}}_{\widehat{\lambda}}](x,i)-\widehat{T}_{\widehat{\lambda}} [\widehat{V}^{\widehat{\mu}}_{\widehat{\lambda}}](x,i)\right|=\left|\max_{a\in\mathcal A}\left\{r_{\widehat{\lambda}}(x,i,a)+\gamma\sum_{(x^\prime, i^\prime)\in\bar{\mathcal X}}P^{\text{aug}}(x^\prime,i^\prime|x,i,a)\widehat{V}^{\widehat{\mu}}_{\widehat{\lambda}}(x^\prime,i^\prime)\right\}\right.\\
&\left.\quad\quad\quad\quad\quad\quad\quad\quad\quad\quad\quad\quad\quad\quad-\max_{a\in\mathcal A}\left\{\widehat{r}_{\widehat{\lambda}}(x,i,a)+\gamma\sum_{(x^\prime, i^\prime)\in\bar{\mathcal X}}\widehat{P}^{\text{aug}}(x^\prime,i^\prime|x,i,a)\widehat{V}^{\widehat{\mu}}_{\widehat{\lambda}}(x^\prime,i^\prime)\right\}\right|\\
\leq & \max_{a\in\mathcal A}\left\{\sum_{x^\prime\in\mathcal X}\left|{P}(x^\prime|x,a)-\widehat{P}(x^\prime|x,a)\right|\frac{ 2\gamma\widehat{\lambda} R_{\max}}{1-\gamma}+\frac{  2\gamma\widehat{\lambda} R_{\max}}{1-\gamma}e(x,a)\right\}\leq\frac{  4\gamma\widehat{\lambda} R_{\max}}{1-\gamma}\max_{a\in\mathcal A} e(x,a).
\end{split}
\]
Thus by combining the above results, the proof is completed.
\end{proof}

Finally, we provide the sub-optimality performance bound in the following theorem. The proof of this theorem is completed by combining the results from Lemma \ref{lem:perf} and \ref{lem:tech}.
\begin{theorem}[Performance]
Let $(\widehat{\mu},\widehat{\lambda})$ be the saddle point solution of $\min_{\lambda\geq 0}\max_{\mu} \widehat{L}(\mu,\lambda)$. Then the following performance bound holds:
\[
\begin{split}
0\leq & \mathbb E_{{P}}\left(\sum_{k=0}^\infty\gamma^k r(x_k,a_k)\mid a_k\sim\mu^*,P_0\right)-\mathbb E_{{P}}\left(\sum_{k=0}^\infty\gamma^k r(x_k,a_k)\mid a_k\sim\widehat{\mu},P_0\right)\\
\leq & \frac{ 2\gamma\widehat{\lambda} R_{\max}}{1-\gamma}\mathbb E_{\widehat{P}}\left[\sum_{k=0}^\infty \gamma^k e(x_k,a_k)|a_k\sim\widehat{\mu}, P_0\right]\\
&+\max_{x\in\mathcal X,a\in\mathcal A}\left\{\frac{ 2(1+\gamma)\gamma\widehat{\lambda} R_{\max}}{(1-\gamma)^2}\mathbb E_{{\widehat{P}}}\left[\sum_{k=0}^\infty \gamma^k e(x_k,a_k)|x_0=x,a_k\sim\widehat{\mu}\right]+\frac{  4\gamma\widehat{\lambda} R_{\max}}{(1-\gamma)^2}e(x,a)\right\}.
\end{split}
\]
\end{theorem}

\section{Conclusions}\label{sec:conclusion}
In this paper, we provided an offline algorithm that solves for a ``good" control policy using a simulated Markov decision process (MDP), and guaranteed that this policy performs better than the baseline policy in the real environment. Furthermore, we provided a performance bound on sub-optimality for the control policy generated by this algorithm.

Future work includes 1) an extension of the problem formulation to include risk-sensitive objective functions, 2) an analysis the convergence rate to the current algorithm, 3) numerical experiments of this algorithm in practical domains and 4) a development of approximate algorithms when the state and control spaces are exponentially huge. 

\newpage
\begin{small}
\bibliography{LTV}
\bibliographystyle{plainnat}
\end{small}

\addtolength{\textheight}{-14cm}   
\end{document}